\documentclass{amsart}

\usepackage{epsfig}
\usepackage{amsmath}
\usepackage{amssymb}
\usepackage{amscd}
\usepackage{graphicx}
\usepackage{color}
\usepackage{float}
\usepackage{verbatim}
\usepackage{bbm}
\usepackage{enumerate}

\newtheorem{thm}{Theorem}[section]

\newtheorem{prop}[thm]{Proposition}

\newtheorem{defn}[thm]{Definition}

\theoremstyle{remark}

\newtheorem{exam}[thm]{Example}

\def \N {\mathbb N}

\def \Z {\mathbb Z}
\def \R {\mathbb R}

\def \F {\mathcal F}

\def \P {\mathcal P}

\def \M {\mathcal M}
\def \mtx {\mathcal M_T(X)}
\def \msy {\mathcal M_S(Y)}
\def \mtxe {\mathcal M^{\mathsf e}_T(X)}

\def \xt {$(X,T)$}
\def \xmt {$(X,\mu,T)$}
\def \yns {$(Y,\nu,S)$}

\def \xtp {$(\bar X,\bar T)$}
\def \mtxp {\mathcal M_{\bar T}(\bar X)}
\def \eps {\varepsilon}

\def \supp {\mathsf{supp}}

\def \sq {sequence}
\def \xsm {$(X,\Sigma,\mu)$}
\def \xsmt {$(X,\Sigma,\mu,T)$}

\def \xsmp {$(\bar X,\bar\Sigma,\bar\mu)$}

\def \xsmtp {$(\bar X,\bar\Sigma_{\bar\mu}, \bar\mu,\bar T)$}

\def \ys {$(Y,S)$}
\def \tl {topological}
\def \im {invariant measure}
\def \inv {invariant}
\def \mpt {measure-preserving system}
\def \ds {dynamical system}

\def \usc {upper semicontinuous}
\def \lsc {lower semicontinuous}

\def \mic {measure-theoretic}
\numberwithin{equation}{section}

\begin{document}

\title{When all points are generic for ergodic measures}

\author{Tomasz Downarowicz and Benjamin Weiss}

\address{\vskip 2pt \hskip -12pt Tomasz Downarowicz}

\address{\hskip -12pt Faculty of Pure and Applied Mathematics, Wroc\l aw University of Technology, Wroc\l aw, Poland}

\email{downar@pwr.edu.pl}

\medskip
\address{\vskip 2pt \hskip -12pt Benjamin Weiss}

\address{\hskip -12pt Einstein Institute of Mathematics,
The Hebrew University of Jerusalem,
Jerusalem, Israel}

\email{weiss@math.huji.ac.il}

\begin{abstract}We establish connections between several properties of \tl\ \ds s, such as: \newline -- every point is generic for an ergodic measure,
\newline -- the map sending points to the measures they generate is continuous,
\newline -- the system splits into uniquely (alternatively, strictly) ergodic subsystems, \newline -- the map sending ergodic measures to their \tl\ supports is continuous,
\newline -- the Ces\`aro means of every continuous function converge uniformly.
\end{abstract}

\maketitle
\section{Introduction}

  It follows from the pointwise ergodic theorem that for a \tl\ \ds\ \xt\ with an ergodic invariant measure $\mu$,  a.e.\ point $x \in X$ 
is generic for $\mu$, i.e.
for any continuous function $f:X\to\R$, the Ces\`aro means
$$
\frac1n\sum_{i=0}^{n-1}f(T^ix)
$$
converge to $\int f\,d\mu$. 
  In 1952, Oxtoby \cite{O} showed that in a uniquely ergodic \tl\ \ds\ \xt\ every point $x\in X$ is generic for the unique \im\ $\mu$ carried by $X$, and moreover the 
convergence, mentioned above, is uniform.  This implies, in particular, that any \tl\ system \xt\ which splits as a disjoint union of uniquely ergodic subsystems, has the property that every point $x\in X$ is generic for an ergodic measure (we will say that \xt\ is \emph{pointwise ergodic}). The converse need not hold: in 1981, Y.\ Katznelson and B.\ Weiss \cite{KW}, provided an example of a pointwise ergodic system with uncountably many ergodic measures, but just one minimal subset (this makes  splitting into uniquely ergodic subsystems impossible).

In 1970, W.\ Krieger, following R.\ Jewett, established that every ergodic \mpt\ has a strictly ergodic (i.e.\ uniquely ergodic and minimal) \tl\ model. It is worth noticing, that once a uniquely ergodic model is found, a strictly ergodic model is easily obtained by taking the (unique) minimal subsystem of that model. Four years later, G.\ Hansel generalized the Jewett--Krieger Theorem to nonergodic systems as follows: every \mpt\ has a \tl\ model \xt\ which splits as a union of strictly ergodic subsystems, moreover,
the Ces\'aro means of any continuous function converge uniformly on $X$ (he called the latter property \emph{uniformity}). It has to be pointed out, that in the nonergodic case, finding a model which splits into uniquely ergodic subsystems is by far insufficient: the minimal subsystems of the uniquely ergodic components need not unite to a closed set, and the uniformity, which holds individually on each uniquely ergodic component, need not coordinate to a global uniformity on~$X$. Thus, a model which splits into uniquely ergodic subsystems neither automatically provides a model which splits into strictly ergodic subsystems, nor one which is uniform. This is why Hansel's result can be considered much stronger than just a straightforward generalization of the Jewett--Krieger Theorem.

We are interested in studying these properties of non-uniquely ergodic \tl\ \ds\ which in the uniquely (alternatively, strictly) ergodic case are automatic:
\begin{itemize}
	\item pointwise ergodicity,
	\item splitting into uniquely ergodic subsystems,
	\item splitting into strictly ergodic subsystems,
	\item uniformity.
\end{itemize} Two more properties emerge naturally in the context of pointwise ergodicity (both of them hold trivially in uniquely ergodic systems):
\begin{itemize}
	\item continuity of the mapping that associates to each point $x\in X$ the ergodic measure for which $x$ is generic, and
	\item continuity of the set-valued function that associates to each ergodic measure  its
	\tl\ support.
\end{itemize}
In this paper we establish a diagram of implications (and equivalences) between these properties (and some of their conjunctions). Our culminating result is proving a characterization of uniformity in terms of the \tl\ organization of the uniquely ergodic components.
\section{Some terminological conventions}
Throughout this paper, by a \emph{\tl\ \ds} we will mean a pair \xt, where $X$ is a compact metric space and $T:X\to X$ is a homeomorphism. By $\M(X)$ and $\mtx$ we will denote the collection of all Borel probability measures on $X$ and of all \inv\footnote{In a \tl\ \ds\ \xt, a measure $\mu$ on $X$ is \emph{\inv} if $\int f\,d\mu = \int f\circ T\,d\mu$, for every continuous $f:X\to\R$.} Borel probability measures on $X$, respectively. By saying ``measure'' we will always mean a Borel probability, since no other measures will be considered.

It is well known that $\mtx$ endowed with the weak-star topology is a nonempty metrizable simplex\footnote{A metrizable \emph{simplex} is a compact convex metric space, such that every point is the barycenter of a unique probability measure supported by the set of extreme points.} and that the extreme points of $\mtx$ are the ergodic measures\footnote{A measure $\mu\in\mtx$ is \emph{ergodic} if all $T$-invariant Borelsets have measure zero or one.}. The set of ergodic measures will be denoted by $\mtxe$. When $\mtx$ is a singleton (equivalently, when $\mtx=\mtxe$), the system \xt\ is called \emph{uniquely ergodic}. A system \xt\ is called \emph{strictly ergodic} if it is uniquely ergodic and minimal\footnote{A \tl\ \ds\ is \emph{minimal} if it contains no proper closed \inv\ subsets.}. It is elementary to see that if \xt\ is uniquely ergodic then $X$ contains a unique minimal subset $M$ which equals the \tl\ support\footnote{Topological support of a measure is the smallest closed set of measure 1.} of the unique \im. Then $(M,T)$\footnote{We write $T$, but in fact, we mean here the restriction of $T$ to $M$.} is a strictly ergodic system and $\mtx=\M_T(M)$.

By a \emph{\mpt} we will understand a quadruple \xsmt, where \xsm\ is a standard (Lebesgue) probability space and $T:X\to X$ is a measure automorphism.
When a \tl\ \ds\ \xt\ is considered with a fixed \im\ $\mu\in\mtx$, it becomes a \mpt\ \xmt. The indication of the sigma-algebra is omitted intentionally, as this role will always be played by the sigma-algebra of Borel sets of $X$ (technically, in order to have a standard probability space, we need to consider the Borel sigma-algebra completed with respect to $\mu$).

Two \mpt s, say \xsmt\ and \xsmtp, are said to be \emph{isomorphic} if there exists an isomorphism $\phi:X\to\bar X$ of the measure spaces \xsm\ and \xsmp, which intertwines the actions of $T$ and $\bar T$, i.e.\ such that $\phi\circ T=\bar T\circ\phi$. In such case, by a commonly accepted abuse of terminology, we will briefly say that $\mu$ and $\bar\mu$ are isomorphic.

Two \tl\ \ds s \xt\ and \xtp\ are \emph{\tl ly conjugate} (briefly \emph{conjugate}) if there exists a homeomorphism $\phi:X\to\bar X$ which intertwines the actions of $T$ and $\bar T$.

A \tl\ \ds\ \xtp\ is called a \emph{model} for a \mpt\ \xsmt, if there exists an \im\ $\bar\mu\in\mtxp$ isomorphic to $\mu$. If \xtp\ is also uniquely ergodic, we call it a \emph{uniquely ergodic model} of \xsmt. If \xt\ is strictly ergodic, we call it a \emph{strictly ergodic} (or \emph{Jewett--Krieger}) \emph{model} of \xsmt. If \xtp\ is a uniquely ergodic model of \xsmt, then $(\bar M,\bar T)$ is a strictly ergodic model of \xsmt,
where $\bar M$ is a unique minimal subset of $\bar X$.

\section{Upper and lower semicontinuous partitions and multifunctions}
Let $X$ be a compact metric space. By $2^X$ we denote the space of all compact subsets of $X$ topologized by the Hausdorff metric (then $2^X$ is also a compact metric space).

\begin{defn}
A partition $\P$ of $X$ with closed atoms is \emph{upper} (resp.\ \emph{lower}) \emph{semicontinuous} if, whenever $(A_n)_{n\ge 1}$ is a \sq\ of atoms of $\P$, converging (in the Hausdorff metric) to a set $S$, then $S$ is contained in (resp.\ contains) an atom of $\P$. If $\P$ is both upper and \lsc, it is called \emph{continuous}.
\end{defn}

The following facts are true. The last statement, for which we give no reference, is an easy exercise.

\begin{prop}\label{quotient}\hfill
\begin{enumerate}
	\item \cite[Theorem 3-31]{HY} Let $f:X\to Y$ be any function between any spaces. Define the \emph{fiber partition}
	as the partition of $X$ whose atoms are the sets $f^{-1}(y)$, $y\in Y$. If $X$ and
	$Y$ are compact metric and $f$ is continuous then the fiber partition has closed atoms and
	is \usc.
	\item \cite[Theorem 3-34]{HY} Any partition $\P$ of a compact
	metric space $X$, with closed atoms and \usc, equals the fiber partition of some
	continuous function $f:X\to Y$ onto a compact metric space $Y$.
	\item The common refinement of any collection of \usc\ partitions with
	closed atoms has closed atoms and is \usc.
\end{enumerate}
\end{prop}

\begin{defn}

A set-valued function $f:Y\to 2^X$, where $X$ and $Y$ are compact metric spaces, is \emph{upper} (resp.\ \emph{lower}) \emph{semicontinuous} if, whenever $(y_n)_{n\ge1}$ is a \sq\ in $Y$ converging to some $y\in Y$ and  $f(y_n)\to S\in 2^X$, then $S$ is contained in (resp.\ contains) $f(y)$.
\end{defn}

It is elementary to see, that a function $f:Y\to2^X$ is continuous if and only if it is both upper and \lsc.

\smallskip
Note the following obvious fact:
\begin{prop}\label{banal}
If $f:Y\to 2^X$, where $X$ and $Y$ are compact metric spaces, is upper (resp.\ lower) semicontinuous, and the collection of values, $\{f(y):y\in Y\}$, happens to be a partition of $X$, then this partition is upper (resp.\ lower) semicontinuous.
\end{prop}

We remark that the converse need not hold. Indeed, suppose that a set-valued function $f:Y\to 2^X$ is such that $\{f(y):y\in Y\}$ is an upper (resp.\ lower) semicontinuous partition of $X$. For any bijection $g:X\to X$, the partition by images of the composition $f\circ g$ is the same as that for $f$, while the composition may destroy any continuity properties of the function.
\smallskip

A very useful class of \lsc\ set-valued functions is provided by the following fact, whose elementary proof we skip:

\begin{prop}Let $X$ be a compact metric space and let $\M(X)$ denote the set of all probability measures on $X$, equipped with the weak-star topology. Then the function $\supp:\M(X)\to 2^X$, assigning to each measure its \tl\ support, is \lsc.
\end{prop}

\section{Strict uniformity and related notions}
Our goal is to introduce and compare several properties, of similar flavor, of \tl\ \ds s. Although none of these properties is new and they appear in various contexts in the literature, to our knowledge they have never been given short names (exception: ``semi-simple''). For convenience of our further discussion, we introduce such names below.
\begin{defn}
A \tl\ \ds\ \xt\ will be called:
\begin{enumerate}[\rm(I)]
	\item \emph{Pointwise ergodic}, if every point $x\in X$ is generic for an ergodic measure.
	In this case, we let $\Phi:X\to\mtxe$ denote the map associating to each point $x\in X$
	the ergodic measure for which $x$ is generic. Also, we let $\P$ denote the fiber partition
	of the function $\Phi$.
	\item \emph{Semi-simple}, if $X$ is a union of minimal sets.
	\item \emph{Partitionable}, if $X$ is a disjoint union of closed \inv\ uniquely ergodic
	sets (we will say \emph{uniquely ergodic subsystems}).
	\item \emph{Strictly partitionable}, if $X$ is a union of strictly ergodic subsystems.
	\item \emph{Continuously pointwise ergodic} if it is pointwise ergodic and 	
	$\Phi$ is continuous.
	\item \emph{Continuously strictly pointwise ergodic} if it is continuously pointwise
	ergodic and semi-simple.
\end{enumerate}
\end{defn}

We begin the analysis of the connections between the notions (I)--(VI) by some fairly obvious observations:

\begin{prop}\label{red}\hfill
\begin{enumerate}
	\item A system \xt\ is partitionable if and only if it is pointwise ergodic and $\P$
	has closed atoms.
	\item A system \xt\ is strictly partitionable if and only if it is pointwise ergodic  and
	semi-simple.
	\item A system \xt\ is continuously pointwise ergodic if and only if it is partitionable
	and $\P$ and is \usc.
	\item A system \xt\ is continuously strictly pointwise ergodic if and only if it is
	strictly partitionable and $\P$ is \usc.
	\item If \xt\ is continuously pointwise ergodic then $\mtxe$ is compact.
\end{enumerate}
\end{prop}

\begin{proof} (1) A partitionable system splits into uniquely ergodic subsystems and each of these subsystems is pointwise ergodic (see \cite{O}). So the whole system is pointwise ergodic as well. The atoms of $\P$ coincide with the uniquely ergodic subsystems, so they are closed. The converse implication is obvious: the atoms of $\P$, if closed, are uniquely ergodic subsystems.

(2) A strictly partitionable system is pointwise ergodic by (1), and semi-simple by definition. If a semi-simple system is pointwise ergodic then so is its every minimal subsystem. A pointwise ergodic minimal system is uniquely ergodic (see \cite[Proposition 3.2]{KW}), which implies strict partitionability.

(3) The fact that continuity of $\Phi$ implies closed atoms and upper semicontinuity of $\P$ follows directly from Proposition \ref{quotient} (1). Suppose that $\P$ has closed atoms and is \usc, and consider a \sq\ $(x_n)_{n\ge1}$ of points in $X$ converging to a point $x\in X$. Denote $\mu_n=\Phi(x_n)$, $\mu=\Phi(x)$, and let $A_n$ and $A$ denote the atoms of $\P$ containing $x_n$ and $x$, respectively. By passing to a sub\sq, we may assume that $\mu_n$ converge to some \im\ $\mu'\in\mtx$ and that $A_n$ converge in $2^X$ to some compact set $A'\subset X$. By upper semicontinuity of $\P$, we have $A'\subset A$. On the other hand, by lower semicontinuity of the set-valued function $\supp$ on $\mtx$, we also have $\supp(\mu')\subset A'$, which implies that $\mu'$ is carried by the atom $A$. Since $A$ carries only one \im\ $\mu=\Phi(x)$, we obtain that $\mu'=\mu$, which concludes the proof of continuity of $\Phi$.

(4) is (3) applied to semi-simple systems.

(5) is obvious: a continuous image of a compact space is compact.
\end{proof}

In addition to the notation (I)--(VI), let us denote by (P) the condition that $\P$ has closed atoms and is \usc, (equivalently, that $\Phi$ is continuous), by~(B) the property that $\mtxe$ is compact (in other words, that $\mtx$ is a Bauer simplex\footnote{A \emph{Bauer simplex} is a simplex whose set of extreme points is closed.}), and by (S) the condition that the set-valued function $\supp$ is continuous on $\mtxe$.

The property (P) has a special interpretation, which follows immediately from Proposition \ref{quotient} (2):

\begin{prop}
Let \xt\ be a pointwise ergodic \ds. The partition $\P$ has closed atoms and is \usc\ if and only if there exists a \tl\ factor $\pi:X\to Y$ from \xt\ to a system \ys\ such that $S$ is the identity on $Y$, and each fiber $\pi^{-1}(y)$ is uniquely ergodic. In this case, for each \im\ $\mu\in\mtx$, $\pi$ coincides (up to measure) with the \mic\ factor map from \xmt\ onto the sigma-algebra of \inv\ sets.
\end{prop}
\begin{comment}
The following diagram shows the relations between our notions, discussed so far:

$$
\begin{matrix}
&&&&\!\!\!\!\!\!\text{(I)}\wedge\text{(II)}\!\!\!\!\!\!&&
\\
&&&&\Updownarrow&&\\
&&\!\!\!\!\!\!\!\!\!\text{(IV)}\wedge\text{(P)}\!\!\!\!\!\!
&&\!\!\!\!\!\!\!\!\!\text{(III)}\wedge\text{(II)}\!\!\!\!\!\!&&
\\
&&\Updownarrow&&\ \ \Updownarrow\scriptstyle{df}&&
\\
\text{(V)}\wedge\text{(II)}\!\!\!&\overset{df}\iff&\text{(VI)}&\implies&\text{(IV)}&\implies&\text{(II)}\\
&&\Downarrow&&\Downarrow&&\\
\text{(III)}\wedge\text{(P)}\!\!\!&\overset{df}\iff&\text{(V)}&\implies&\text{(III)}&\implies&\text{(I)}\\
&&\Downarrow&&&&\\
&&\text{(B)}&&
\end{matrix}
$$

\end{comment}

We continue by establishing more equivalences and implications.

\begin{prop}\label{phitos}\emph{(VI)$\iff$(IV)$\wedge$(B)$\wedge$(S)}.\hfill\break
A system \xt\ is continuously strictly pointwise ergodic if and only if it is strictly partitionable, $\mtxe$ is compact, and the set-valued function $\supp:\mtxe\to2^X$ is continuous. The partition $\P$ is then continuous.
\end{prop}

\begin{proof}
Assume that \xt\ is continuously strictly pointwise ergodic. We already know that then $\mtxe$ is compact. Let $(\mu_n)_{n\ge1}$ be a \sq\ of ergodic measures converging to an ergodic measure $\mu$. Let $A_n$ and $A$ denote the \tl\ supports of $\mu_n$ and $\mu$, respectively. Consider a \sq\ of points $x_n\in A_n$ converging to a point $x\in X$. Since $\Phi$ is continuous, the measures $\mu_n=\Phi(x_n)$ converge to $\mu=\Phi(x)$, i.e.\ $x$ is generic for $\mu$. Since \xt\ is strictly partitionable, this implies that $x\in\supp(\mu)=A$. We have shown that $\supp$ is \usc\ on $\mtxe$. Since $\supp$ is \lsc\ on $\M(X)$, we have proved continuity of $\supp$ on $\mtxe$.

Conversely, assume that \xt\ is strictly partitionable, $\mtxe$ is compact, and $\supp$ is continuous on $\mtxe$. Then $\P$ equals the partition by images of the continuous set-valued function $\supp$ defined on a compact space, so it is continuous, in particular \usc. Proposition \ref{red} (4) now implies that \xt\ is continuously strictly pointwise ergodic.
\end{proof}

We are interested in two other properties of \tl\ \ds s, which have seemingly a much more ``ergodic-theoretic'' flavor than all the ``\tl'' properties discussed so far. Nonetheless, we will be able to provide their \tl\ charactarizations.
Both notions have been introduced in \cite{H}, but the respective properties seem to originate in a much earlier work \cite{O}.

\begin{defn}\label{uni}\hfill
\begin{itemize}
\item A \tl\ \ds\ \xt\ is called \emph{uniform} if, for every continuous function $f:X\to\R$, the Ces\`aro means
$$
A^f_n(x)=\frac1n\sum_{i=0}^{n-1}f(T^ix)
$$
converge uniformly on $X$.
\item A system \xt\ which is both uniform and semi-simple is called \emph{strictly uniform}.
\end{itemize}
\end{defn}

Clearly, in a uniform system, the limit function $\tilde f=\lim_n A^f_n$ is continuous.
As we have already mentioned several times, uniquely ergodic systems are uniform. We have the following nearly obvious implication:

\begin{prop}\label{un}
Any uniform system is continuously pointwise ergodic. Any strictly uniform system is continuously strictly pointwise ergodic.
\end{prop}

\begin{proof}
Clearly, a uniform system \xt\ is pointwise ergodic. It is now crucial to observe that the partition $\P$ equals the common refinement over all continuous function $f:X\to\R$ of the fiber partitions of the limit functions~$\tilde f$. Now, by Proposition \ref{quotient} (1) and (3), $\P$ is \usc, and by
Proposition \ref{red} (3), \xt\ is continuously pointwise ergodic. The statement about strictly uniform systems follows from the above applied to semi-simple systems.
\end{proof}

The second part of Proposition \ref {un} places strict uniformity at the top of the hierarchy of properties discussed in this paper, as the most restrictive one.
A notable role of strict uniformity in the theory of dynamical systems is sanctioned by two theorems, the first of which is the celebrated Jewett--Krieger Theorem (\cite{J} and \cite{K}), the second, although vastly generalizes the first one, is much less commonly known.

\begin{thm}{\rm (Jewett--Krieger\footnote{Jewett showed this assuming weak mixing, Krieger relaxed this assumption to ergodicicty.}, 1970)}
Every ergodic \mpt\ \xsmt\ has a strictly ergodic \tl\ model.
\end{thm}

As a corollary, we obtain that
\begin{itemize}
	\item \emph{every ergodic \mpt\ has a strictly uniform \tl\ model}.
\end{itemize}

The second theorem is due to G.\ Hansel \cite{H}, and generalizes the above to non-ergodic \mpt s:
\begin{thm}{\rm (Hansel, 1974)}
Every \mpt\ \xsmt\ has a strictly uniform \tl\ model.
\end{thm}

Hansel's theorem has the following remarkable interpretation: the class of strictly uniform systems, although the smallest (among the classes discussed in this paper), is still reach enough to capture ``the entire ergodic theory''.
\smallskip

We will denote uniformity by (U) and strict uniformity by (sU). The following theorem, which to our best knowledge has not been noted up to date, constitutes the main result of this paper. Combined with the equivalence (V)$\iff$(IV)$\wedge$(P), it characterizes the ``ergodic'' property of uniformity by means of ``topological organization'' of the uniquely ergodic subsystems.

\begin{thm}\label{ben}\emph{(U)$\iff$(V), (sU)$\iff$(VI)}.\hfill\break
A system \xt\ is uniform if and only if it is continuously pointwise ergodic. It is strictly uniform if and only if it is continuously strictly pointwise ergodic.
\end{thm}

\begin{proof} In view of Proposition \ref{un}, we only need to prove that continuously pointwise ergodic systems are uniform. So, assume that \xt\ is continuously pointwise ergodic. Fix any continuous function $f:X\to[0,1]$ and a point $x_0\in X$. Let $\mu_0=\Phi(x_0)$. Given $\eps>0$, there exists an $n_0=n(x_0)$ such that
$$
\left|A_{n_0}^f(x_0)-\int f\,d\mu_0\right|<\frac\eps2.
$$
The $|A_{n_0}^f(x)-\int f\,d\mu|$ is a doubly continuous function of $(x,\mu)$, so the inequality $|A_{n_0}^f(x)-\int f\,d\mu|<\frac\eps2$ holds for all measures $\mu$ in a sufficiently small neighborhood $V_0$ of $\mu_0$ and all points $x$ is a sufficiently small neighborhood $U_0$ of $x_0$. By continuity of $\Phi$, we can choose $U_0$ so small that $\Phi(U_0)\subset V_0$. Then we have
$$
\left|A_{n_0}^f(x)-\int f\,d\mu_{x}\right|<\frac\eps2,
$$
for all $x\in U_0$, where $\mu_{x}=\Phi(x)$. By compactness, $X$ is covered by finitely many neighborhoods $U_i$ created analogously (with the same $\eps$) for some points $x_i\in X$ ($i=1,2,\dots,l$). This cover defines finitely many numbers $n_i = n(x_i)$. Let $n$ be any integer larger than
$$
M=\frac{2\max\{n_1,\dots,n_l\}}\eps.
$$
Pick any $x\in X$. We can divide the forward orbit of $x$ of length $n$ into ``portions'', as follows: $x$ belongs to some $U_{i_0}$, so we choose $n_{i_0}$ as the length of first portion. Next, $T^{n_{i_0}}x$ belongs to some $U_{i_1}$, so we choose
$n_{i_1}$ as the length of second portion. Next, $T^{n_{i_0}+n_{i_1}}x$ belongs to some $U_{i_2}$, so we choose $n_{i_2}$ as the length of third portion. And so on, until we are left with a ``tail'' of some length of length $m<\max\{n_1,\dots,n_l\}$, starting at $T^{n_{i_0}+n_{i_1}+\cdots+n_{i_q}}x$ for some $q\in\N$. We have
$$
A_n^f(x) = \frac{n_{i_0}}n A_{n_{i_0}}x+\sum_{j=1}^q\frac{n_{i_j}}n A_{n_{i_j}}^f(T^{n_{i_0}+n_{i_1}+\cdots+n_{i_{j-1}}}x)+\frac mn A_m^f(T^{n_{i_1}+n_{i_2}+\cdots+n_{i_q}}x).
$$
Because all the points $T^{n_{i_j}}x$ are generic for the same measure $\mu_x$, we have
$$
\left|A_{n_{i_0}}^f(x)-\int f\,d\mu_x\right|<\frac\eps2,
$$
and for each $j=1,2,\dots,q$,
$$
\left|A_{n_{i_j}}^f(T^{n_{i_0}+n_{i_1}+\cdots+n_{i_{j-1}}}x)-\int f\,d\mu_x\right|<\frac\eps2,
$$
while the last term $\frac mn A_m^f(T^{n_{i_1}+n_{i_2}+\cdots+n_{i_q}}x)$ has absolute value smaller than $\frac\eps2$. By averaging, we obtain that the inequality
$$
\left|A_n^f(x)-\int f\,d\mu_x\right|<\eps
$$
holds for any $n>M$ uniformly for all $x\in X$, which implies uniformity.

The equivalence between strict uniformity and continuous strict pointwise ergodicity is the (just proved) equivalence between continuous pointwise ergodicity and uniformity, applied to semi-simple systems.
\end{proof}

The following diagram shows all the implications between the conditions discussed in this paper: (I)--(VI), (U), (sU), (S), (P) and (B):
\smallskip

$$
\begin{matrix}
&&\!\!\!\!\!\!\!\!\!\!\!\!\!\!\!\!\!\!\text{(IV)}\wedge\text{(B)}\wedge\text{(S)}\!\!\!\!\!\!\!\!\!\!\!\!\!\!&&\!\!\!\!\text{(I)}\wedge\text{(II)}\!\!\!\!\!&&
\\
&&\Updownarrow&&\Updownarrow&&
\\
&&\!\!\!\!\!\!\!\!\!\text{(IV)}\wedge\text{(P)}\!\!\!\!\!\!&&\!\!\!\!\!\!\!\!\!\!\!\!\!\text{(III)}\wedge\text{(II)}\!\!\!\!\!\!\!\!\!\!&&
\\
&&\Updownarrow&&\Updownarrow\scriptstyle{df}\!\!\!\!\!&&
\\
\text{(V)}\wedge\text{(II)}\!\!\!&\overset{df}\iff&\text{(VI=sU)}&\implies&\text{(IV)}&\implies&\text{(II)}
\\
&&\Downarrow&&\Downarrow&&\\
\text{(III)}\wedge\text{(P)}\!\!\!&\overset{df}\iff&\text{(V=U)}\, &\implies&\text{(III)}&\implies&\text{(I)}\\
&&\Downarrow&&&&\\
&&\text{(B)}&&&&
\end{matrix}
$$
\smallskip

By analogy to the conditions equivalent to (VI), one might ask: is (V) equivalent to (III)$\wedge$(B)$\wedge$(S)? This question has a negative answer: as we will show in examples, neither (III)$\wedge$(B)$\wedge$(S) implies (V) nor (V) implies (S).

\smallskip
Nonetheless, the conjunction (III)$\wedge$(B)$\wedge$(S) has a special meaning. Recall that in a uniquely ergodic system \xt, the unique minimal set $M$ produces a strictly ergodic (hence strictly uniform) subsystem $(M,T)$ carrying the same \im\ as \xt. Notice that if \xt\ is partitionable then the union of all minimal sets (regardless of whether it is closed or not) supports the same \im s as \xt. Thus, a natural question arizes: under what conditions does the union of minimal subsets of a partitionable system \xt\ produce a strictly uniform subsystem. Below we give a complete answer.

\begin{thm}\label{cut}% \emph{(III)$\wedge$(B)$\wedge$(S)$\iff(M,T)$ satisfies (VI)}.
Let \xt\ be partitionable. Let $M$ denote the union of all minimal subsets of $X$. Then $M$ is closed and the system $(M,T)$ is strictly uniform if and only if $\mtxe$ is compact and the set-valued function $\supp:\mtxe\to2^X$ is continuous.
\end{thm}

\begin{proof}
First notice that partitionability of \xt\ implies that the family of all minimal subsets coincides with $\supp(\mtxe)$, i.e.\ with the image of $\mtxe$ via $\supp$. Now assume that $\mtxe$ is compact and $\supp:\mtxe\to2^X$ is continuous. It follows that $\supp(\mtxe)$ is compact, hence closed in $2^X$. This easily implies that the union $M$ of all minimal sets is closed in $X$. The system $(M,T)$ satisfies partitionability (III) and semi-simplicity (II), so it is strictly pointwise ergodic (IV). It also satisfies (B) and (S). Since (IV)$\wedge$(B)$\wedge$(S)$\iff$(sU), $(M,T)$ is strictly uniform.

Conversely, assume that $M$ is closed and that the system $(M,T)$ is strictly uniform. Then $(M,T)$ satisfies (B) and (S), i.e.\ $\M^{\mathsf e}_T(M)$ is compact and the function $\supp:\M^{\mathsf e}_T(M)\to2^M\subset 2^X$ is continuous. By partitionability, we have $\M^{\mathsf e}_T(M)=\mtxe$, implying that the conditions (B) and (S) hold also for \xt.
\end{proof}

We remark, that without partitionability of \xt\ (even with pointwise ergodicity assumed instead), Theorem \ref{cut} loses its valuable interpretation. The conjunction (B)$\wedge$(S) still implies that $M$ is compact and whenever $(M,T)$ is pointwise ergodic, it is strictly uniform. However, without partitionability, there is no guarantee that $\M_T(M)=\mtx$ and $(M,T)$ cannot replace \xt\ in the role of a \tl\ model of some \mpt s. This happens for instance in the example given by Katznelson and Weiss in \cite{KW} (see below for more details). 

On the other hand, partitionability of \xt\ does not follow from the assumptions that \xt\ is pointwise ergodic, $M$ is compact, even if $(M,T)$ is strictly uniform and carries the same \im s as \xt. The simplest example here is $X$ being the two-point compactification of $\Z$, $X=\Z\cup\{-\infty,\infty\}$, with $Tx=x+1$. In this system $M=\{-\infty,\infty\}$ is closed and $(M,T)$ is trivially strictly uniform. Nonetheless, \xt\ is pointwise ergodic without being partitionable; the set of points generic for $\delta_\infty$ equals $\Z\cup\{\infty\}$ and is not closed.

\section{Examples} One of the first examples showing phenomena associated to the notions discussed in the preceding paragraph is given in \cite{KW}. The system \xt\ is pointwise ergodic but not partitionable. The set $\mtxe$ is homeomorphic to the unit interval by a map $t\mapsto\mu_t$, in particular $\mtxe$ is compact and uncountable. The \tl\ supports of the ergodic measure form a nested family of sets $E_t$ intersecting at a fixpoint $x_0$ (which is a unique minimal subset). It can be verified that also the set-valued map $t\mapsto E_t$ is continuous (into $2^X$), implying that $\supp$ is continuous on $\mtxe$. This example shows that (I)$\wedge$(B)$\wedge$(S)$\,\ \not\!\!\!\implies$(III).
\smallskip

We continue with six other examples. They are elementary, yet they resolve possible questions about other implications. The space $X$ in the first five of them is the same and consists of a \sq\ of circles $C_k$ ($k\ge 1$) converging to a circle $C$ (in the last example we add to this space one more circle $C'$). The circles $C_k$ and $C$ can be imagined as all having equal diameters and placed on parallel planes, or as concentric circles with radii of $C_k$ converging to the radius of $C$. On each circle $C_k$ we pick one point $c_k$ so that these points converge to a point $c\in C$. We will call the points $c_k$ and $c$ the \emph{origins} of the circles $C_k$ and $C$, respectively (they will be used in Examples \ref{ex2}, \ref{ex3} and \ref{ex6}).

\begin{exam}\label{ex1}
In this example, the transformation $T$ restricted to each circle $C_k$ (henceforth denoted by $T_k$) is an irrational rotation. The angles of the rotations $T_k$ tend to zero with $k$, so that $T$ restricted to $C$ is the identity map (see Figure \ref{fig1}).
\begin{figure}[h]
\includegraphics[width=12cm]{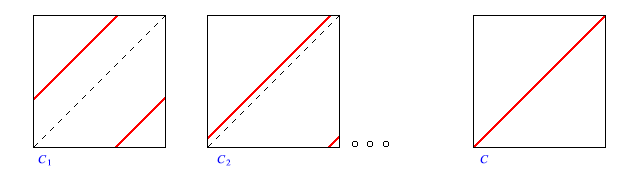}
\caption{}\label{fig1}
\end{figure}

Clearly, this system is strictly partitionable. The strictly ergodic subsystems are the circles $C_k$ and all single points on $C$. The set $\mtxe$ consists of the normalized Lebesgue measures $\lambda_k$ on the circles $C_k$ and the Dirac measures $\delta_x$ at all points $x\in C$. This set of measures is not closed in $\mtx$ (the measures $\lambda_k$ converge to the non-ergodic Lebesgue measure on $C$). So, the system  cannot be continuously pointwise ergodic and, by Theorem \ref{ben}, it is not uniform. It follows, that the partition $\P$ connot be \usc, and $\Phi:X\to\mtx$ cannot be continuous (all these negative properties can be seen directly, but we want show how our diagram works). On the other hand, the map $\supp$ is easily seen to be continuous, even on the closure of $\mtxe$, which consists of two disjoint compact sets: one containing the Lebesgue measures on all the circles (including $C$) and second, consisting of all Dirac measures at the points $x\in C$. This example shows that without (B), (IV)$\wedge$(S)$\,\ \not\!\!\!\implies$(VI) (not even (V)).
\end{exam}

\begin{exam}\label{ex2}
In the previous example we slightly change the transformations $T_k$. For each $k$ we pick an arc $I_k$ containing the origin $c_k$. The diameters of the arcs decrease to zero, so that the arcs $I_k$ converge to the singleton $\{c\}$. On each $I_k$ we ``slow down the motion'', meaning that we locally bring the graph of $T_k$ closer to the diagonal (see Figure \ref{fig2}). Since $T$ is not changed outside the arcs $I_k$, (and on each $I_k$ the map is even closer to identity than before), the maps $T_k$ converge, as before, to the identity map on $C$.  We can easily arrange the modifications so that each $T_k$ is conjugate to some irrational rotation (much slower than before). We let $\mu_k$ denote the unique \inv\ (and ergodic) measure supported by $C_k$. Each orbit in $C_k$ now spends much more time in $I_k$ than before, so $\mu_k(I_k)$ is larger than $\lambda_k(I_k)$. By slowing down strongly enough we can easily arrange that $\mu_k(I_k)\to 1$, implying that the measures $\mu_k$ converge to $\delta_c$.
\begin{figure}[h]
\includegraphics[width=12cm]{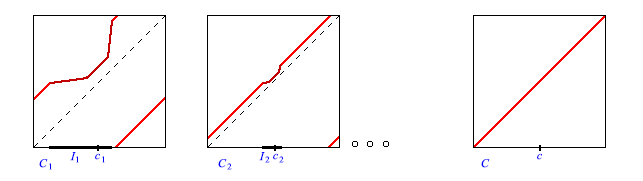}
\caption{}\label{fig2}
\end{figure}

The system is strictly partitionable with the partition $\P$ the same as in the previous example. Since this partition is not \usc, the system is not continuously pointwise ergodic. However, this time, the set of ergodic measures $\{\mu_k:k\ge 1\}\cup\{\delta_x:x\in C\}$ is closed in $\mtx$ (because $\mu_k\to\delta_c$), hence compact. Since (IV)$\wedge$(B) holds and (VI) does not, it must be (S) that fails (this can also be seen directly:  $\supp(\mu_k)=C_k\to C\neq\supp(\delta_c)$). This example shows that without (S), (IV)$\wedge$(B)$\,\ \not\!\!\!\implies$(VI) (not even (V)).
\end{exam}

\begin{exam}\label{ex3}This is the declared example showing that (III)$\wedge$(S)$\wedge$(B)$\,\ \not\!\!\!\implies$(V). Let the map $T_k$ on the circle $C_k$ fix the origin $c_k$ and move all other points, say, clockwise, but with the speed of the movement decaying as $k$ increases, so that on $C$ we have the identity map (see figure \ref{fig3}).

\begin{figure}[h]
\includegraphics[width=12cm]{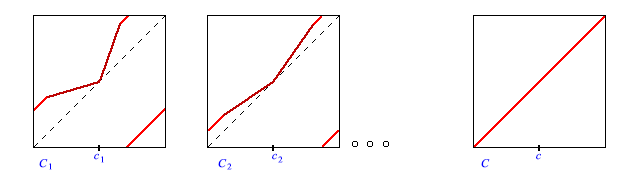}
\caption{}\label{fig3}
\end{figure}

Clearly, the system is partitionable, with the same partition $\P$ as in Examples \ref{ex1} and \ref{ex2}. Since $\P$ is not \usc, the system is not continuously pointwise ergodic.
The set of ergodic measures equals $\{\delta_{c_k}:k\ge1\}\cup\{\delta_x:x\in C\}$ and is compact, and the set-valued function $\supp:\mtxe\to2^X$ is obviously continuous.
\end{exam}

\begin{exam}\label{ex4}This is the other declared example showing that (V)$\,\ \not\!\!\!\implies$(S) (and some other failing implications).
The maps $T_k$ on the circles $C_k$ are modifications of one fixed irrational rotation. We can ``slow down the motion'' on each circle $C_k$ so that the resulting maps $T_k$ are conjugate to irrational rotations, and converge to a continuous map on $C$ which fixes just one point $c$ and moves all other points, say, clockwise (see Figure \ref{fig4}). Then $\delta_c$ is the unique \inv\ (and ergodic) measure supported by $C$. The unique invariant (and ergodic) measures $\mu_k$ supported by $C_k$ must accumulate at \im s supported by $C$, so they have no other choice than to converge to $\delta_c$.
\begin{figure}[h]
\includegraphics[width=12cm]{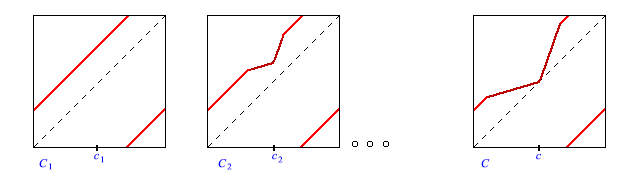}
\caption{}\label{fig4}
\end{figure}

The the system is partitionable but not semi-simple. The partition $\P$ into closed uniquely ergodic sets equals $\{C_k:k\ge 1\}\cup\{C\}$ and is continuous, so \xt\ is continuously pointwise ergodic. In particular, $\mtxe$ is compact (the last condition can also be seen more directly, by noting that $\mtxe=\{\mu_k:k\ge 1\}\cup\{\delta_c\}$). However, since the supports of $\mu_k$ are whole circles, while that of the limit measure $\delta_c$ is a singleton, the function $\supp:\mtxe\to 2^X$ is not continuous. This example shows that (V)$\,\ \not\!\!\!\implies$(II) and (V)$\,\ \not\!\!\!\implies$(S). This example also shows that although (VI) implies continuity of $\P$, (V) combined with this continuity still does not imply~(VI).
\end{exam}

\begin{exam}\label{ex5}Now we change the maps $T_k$ so that they all are the same as that on $C$ (see figure \ref{fig5}).

\begin{figure}[h]
\includegraphics[width=12cm]{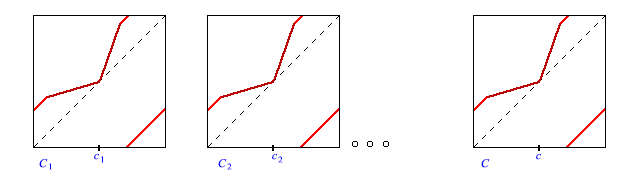}
\caption{}\label{fig5}
\end{figure}

Like in the previous example, the system is partitionable but not semi-simple, with $\P$ continuous, hence \xt\ is continuously pointwise ergodic. This time
the map $\supp:\mtxe\to2^X$ is continuous.  The example shows that (V) combined with (S) (and additionally with the continuity of $\P$) still does not suffice for (VI).
\end{exam}

\begin{exam}\label{ex6}In the above example, we can add to the space one more circle $C'$ that is tangent to $C$ at $c$. The dynamics on $C'$ is symmetric to that on $C$, so that all points on $C'$ are generic for $\delta_c$ and $C\cup C'$ is an atom of $\P$ (see figure \ref{fig6}).

\begin{figure}[h]
\includegraphics[width=5cm]{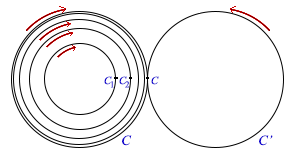}
\caption{}\label{fig6}
\end{figure}

So modified system has the same properties as in the preceding example, except that the partition $\P$ is now \usc\ not continuous. This example shows that (V) (contrary to (VI)) does not imply continuity of $\P$, even when combined with (S).
\end{exam}

\section*{Acknowledgements}
The authors thank Mariusz Lema\'nczyk for raising a question that triggered this research.

Tomasz Downarowicz is supported by
National Science Center, Poland (Grant No. 2018/30/M/ST1/00061) and
by the Wroc\l aw University of Science and Technology (Subsidy for 2020, budget no. 8201003902 MPK: 9130730000).

\end{document}